\documentclass[12pt]{elsarticle}
\usepackage[utf8]{inputenc}
\usepackage{amsmath}
\usepackage{amssymb}
\usepackage{newtxtext}
\usepackage{newtxmath}
\usepackage{microtype}
\usepackage{xcolor}
\usepackage{graphicx}
\usepackage{enumerate}
\usepackage{bbold}
\usepackage{xfrac}
\makeatletter
\def\ps@pprintTitle{%
 \let\@oddhead\@empty
 \let\@evenhead\@empty
 \def\@oddfoot{\centerline{\thepage}}%
 \let\@evenfoot\@oddfoot}
\makeatother

\newtheorem{thm}{Theorem}[section]

\newtheorem{lemma}{Lemma}[section]

\newtheorem{prop}{Proposition}[section]

\newtheorem{conjecture}{Conjecture}[section]

\usepackage{amssymb}
\begin{document}
\begin{frontmatter}

\title{Projection Constants and Banach-Mazur Distance in Codimension Two}

 \author[label1]{Beata~Deregowska}
 
 \author[label4]{Barbara~Lewandowska \footnote{B.L. is partially supported by National Science Center (NCN) grant no. 2025/09/X/ST1/01604 }}

 \address[label1]{Institute of Mathematics\\
University of the National Education Commission, Podchorazych~2, Krakow, 30-084, Poland}
\address[label4]{Faculty of Mathematics and Computer Science\\
Jagiellonian University, Lojasiewicza~6, Krakow, 30-348, Poland}

%% \author{llll}

%% \affiliation{organization={},%Department and Organization

%\date{June 2022}
\begin{abstract}We prove that every subspace $V\subset\ell_\infty^N$
of codimension one or two with absolute projection constant
$\lambda(V)<4/3$ satisfies
$$
d(V,\ell_\infty^{\dim V})
\leq\frac{\lambda(V)}{4-3\lambda(V)}.
$$
We conjecture that the codimension assumption can be removed and that
the same estimate holds for every subspace $V\subset\ell_\infty^N$ with
$\lambda(V)<4/3$.

\end{abstract}




\end{frontmatter}
\section{Introduction}
Throughout the paper, all Banach spaces are real. Let $V$ be a nonzero
finite-dimensional subspace of a Banach space $X$. A bounded linear map
$P:X\to V$ is a projection onto $V$ if $P|_V=\mathrm{id}_V$.
The relative projection constant is
$$
\lambda(V,X)=\inf\{\|P\|:P:X\to V\text{ is a projection}\}.
$$
The absolute projection constant of $V$ is defined by
$$
\lambda(V)=\sup\{\lambda(I(V),X): I:V\to X
\text{ is an isometric embedding}\},
$$
where the supremum is taken over all Banach spaces $X$ and all such
embeddings. It is known that
$\lambda(V)=\lambda(I(V),\ell_\infty)$ for any isometric copy $I(V)$
of $V$ in $\ell_\infty$ (see e.g. \cite{Wojtaszczyk}).
In particular, if $V\subset\ell_\infty^N$, then
\begin{equation}\label{absrel}
\lambda(V)=\lambda(V,\ell_\infty^N).
\end{equation}

For two finite-dimensional Banach spaces $V$ and $W$ of the same
dimension, their Banach-Mazur distance is
$$
d(V,W)=\inf\{\|T\|\,\|T^{-1}\|:T:V\to W\text{ is an isomorphism}\}.
$$
If $\dim V=m$, the injectivity of $\ell_\infty^m$ gives
\begin{equation}\label{lambdaBM}
\lambda(V)\leq d(V,\ell_\infty^m).
\end{equation}
Indeed, every isomorphism $T:V\to\ell_\infty^m$ extends to any Banach
superspace of $V$ without increasing its norm; composing such an
extension with $T^{-1}$ gives a projection onto $V$. The classical
characterization of finite-dimensional injective spaces shows that
$\lambda(V)=1$ if and only if $d(V,\ell_\infty^m)=1$; see
\cite{Nachbin,Goodner,Kelley}. For early results on projections in
classical Banach spaces, see also \cite{Pelczynski1960}.

It is natural to ask whether \eqref{lambdaBM} admits a converse estimate
independent of the dimension. More precisely, does there exist a
nondecreasing function $\Phi:[1,\infty)\to[1,\infty)$, with $\Phi(1)=1$,
such that
\begin{equation}\label{Plambda}
d(V,\ell_\infty^m)\leq\Phi(\lambda(V))
\end{equation}
for every $m$-dimensional Banach space $V$? This question is discussed
by Zippin in \cite{ZippinRange} and is closely related to the classical
$P_\lambda$-problem; see \cite{LiPe}. Here a $P_\lambda$-space is a
Banach space whose every isometric copy in a Banach superspace is the
range of a projection of norm at most $\lambda$. For finite-dimensional
$V$, this is equivalent to $\lambda(V)\leq\lambda$.

A global linear bound in \eqref{Plambda} is impossible. Szarek
\cite{Szarek} constructed, for arbitrarily large $m$, $m$-dimensional
spaces $V$ such that
$$
d(V,\ell_\infty^m)\geq c\sqrt m\log m
$$
with an absolute constant $c>0$. Together with the Kadets-Snobar bound
$\lambda(V)\leq\sqrt m$ \cite{KadetsSnobar}, this rules out an estimate
$d(V,\ell_\infty^m)\leq C\lambda(V)$ with an absolute constant $C$.
On the other hand, Zippin proved that projection constants sufficiently
close to one force the distance to the corresponding $\ell_\infty$-space
to be close to one, independently of the dimension
\cite{ZippinPlambda,ZippinCorrection}. More precisely, there exist
$\varepsilon_0>0$ and a function
$\varphi:(0,\varepsilon_0)\to[0,\infty)$, with
$\varphi(\varepsilon)\to0$ as $\varepsilon\to0$, such that
$$
d(V,\ell_\infty^{\dim V})\leq1+\varphi(\varepsilon)
$$
for every $0<\varepsilon<\varepsilon_0$ and every finite-dimensional
space $V$ with $\lambda(V)\leq1+\varepsilon$.
Related results concerning large subspaces of ranges of small-norm
projections and finite-dimensional $P_\lambda$-spaces were obtained in
\cite{ZippinRange,Bourgain}. Further estimates under structural
assumptions appear in \cite{Schutt,Sass}; for bounds based on
factorization methods, see \cite{BourgainSzarek}.
The aim of this paper is to obtain an explicit estimate for the
Banach-Mazur distance to an $\ell_\infty$-space in terms of the
projection constant. We prove that every nonzero subspace
$V\subset\ell_\infty^N$ of codimension one or two with
$\lambda(V)<4/3$ satisfies
\begin{equation}\label{introBM}
d(V,\ell_\infty^{\dim V})
\leq
\frac{\lambda(V)}{4-3\lambda(V)}.
\end{equation}
Our main result, Theorem~\ref{Bazur}, establishes this estimate
for subspaces of codimension two. The same argument, combined
with Proposition~\ref{hypergeom}, yields the hyperplane case.
Equivalently, writing $\lambda(V)=1+\varepsilon$, we obtain
$$
d(V,\ell_\infty^{\dim V})
\leq
\frac{1+\varepsilon}{1-3\varepsilon},
\qquad
0\leq\varepsilon<\frac13.
$$

For comparison, a quantitative form of Zippin's general result,
stated in \cite[p.~809]{GKL} with reference to Zippin and Bernau,
provides an estimate for arbitrary finite-dimensional spaces when
$0<\varepsilon<17^{-8}$.
For hyperplanes and codimension-two subspaces, our estimate covers
the much larger interval $0\leq\varepsilon<1/3$.
Moreover, its right-hand side is
$1+4\varepsilon+O(\varepsilon^2)$ as $\varepsilon\to0$.

The proof in codimension two combines an elementary rank estimate
with a geometric property of two-dimensional subspaces of
$\ell_\infty^4$. Together, these allow us to construct a basis of
$V$ from signed sums of disjoint groups of columns of a small-norm
projection and to obtain \eqref{introBM}.
In the final section, we show that this geometric property fails
in higher codimension. For the counterexamples constructed there,
however, the required condition on groups of columns still holds
Thus these examples show a limitation of the geometric argument,
rather than a failure of the Banach-Mazur estimate,
and we expect the latter
to hold in general.

\begin{conjecture}\label{coniektore}
Every nonzero subspace $V\subset\ell_\infty^N$ with $\lambda(V)<\frac43$
satisfies
$$
d(V,\ell_\infty^{\dim V})\leq\frac{\lambda(V)}{4-3\lambda(V)}.
$$
\end{conjecture}
Since every finite-dimensional Banach space is a limit, in the Banach--Mazur
distance, of subspaces of the spaces $\ell_\infty^N$, and since
$\lambda(W)\leq d(V,W)\,\lambda(V)$ for all $V,W$ of the same dimension,
Conjecture~\ref{coniektore} is equivalent to the same statement for all
finite-dimensional Banach spaces.\\
~\\
{\bf Notation.}
For integers $a,b$, we write $[a:b]=\{a,a+1,\ldots,b\}$ when
$a\leq b$, and set $[a:b]=\varnothing$ otherwise. The standard canonical vectors are denoted by $e_1,\ldots,e_N$;
we use the same notation for the corresponding coordinate functionals.
For a matrix $A=(a_{ij})$, we denote by $A_{I,J}$ the submatrix
formed by the rows indexed by $I$ and the columns indexed by $J$.
The $j$-th column of a matrix $P=(p_{ij})$ is denoted by $p_j$.
We use $\|\cdot\|_1$ and $\|\cdot\|_\infty$ for the usual
$\ell_1$ and $\ell_\infty$ norms, and $\|\cdot\|$ for operator norms.\\
For a subspace $V\subset\mathbb R^N$, we denote its orthogonal
complement with respect to the standard Euclidean inner product
by $V^\perp$.
Under the identification $(\ell_\infty^N)^*=\ell_1^N$,
this is also the annihilator of $V$.\\
Finally, for linearly independent $f,g\in\ell_1^4$, we set
$$
\Delta_i=f_i g-g_i f,
\qquad
\Delta_{ij}=f_i g_j-g_i f_j=(\Delta_i)_j,
\qquad i,j\in[1:4].
$$
Thus $\Delta_{ji}=-\Delta_{ij}$, $\Delta_{ii}=0$, and
$\|\Delta_i\|_1=\sum_{j=1}^4|\Delta_{ij}|$.

\section{Main result}
We first recall the classical L\'evy-Desplanques theorem.
We then give a proof of an elementary rank estimate under a weaker
diagonal dominance assumption. This estimate may well be folklore,
but we have not found it stated explicitly in standard textbooks
on linear algebra or matrix analysis. We therefore include a proof.
\begin{thm}[L\'evy-Desplanques, see e.g.\ {\cite{HJ}}]\label{domina}
    Let $A=(a_{ij})\in\mathbb{K
    }^{n\times n}$, where $\mathbb{K}=\mathbb{R}$ or $\mathbb{K}=\mathbb{C}.$ If $A$ is strictly diagonally dominant, that is,
$$ 
|a_{ii}|>\sum_{j\ne i}|a_{ij}|, \qquad i\in [1:n] 
$$
then $A$ is nonsingular.
\end{thm}
\begin{lemma}\label{rankus}
Let $A=(a_{ij})\in \mathbb{K}^{N\times N},$ where $\mathbb{K}=\mathbb{R}$ or $\mathbb{K}=\mathbb{C}.$ If 
$$
2|a_{ii}|>\sum_{j\ne i}|a_{ij}|, \qquad i\in [1:N] 
$$
then ${\rm rank}(A)\geq \tfrac{N}{2}.$
\end{lemma}
{\sc Proof.} Suppose, to the contrary, that $r:=\mathrm{rank}(A)<\tfrac{N}{2}$. The assumption implies $a_{ii}\neq 0$ for all $i$, so $A\neq 0$ and $r\geq 1$. Among all $r\times r$ submatrices of $A$, choose one, say $A_{I,J}$, with maximal $|\det A_{I,J}|$. Since $\mathrm{rank}(A)=r$, we have $\det A_{I,J}\neq 0$.  In particular, the columns $a_j$, $j\in J$, of $A$  form a basis of the column space of $A$. Hence, for every $k\in J':=[1:N]\setminus J$, there exist scalars $c_j^k\in \mathbb{K}$, $j\in J$, such that
\begin{equation}\label{kombin}
a_k=\sum_{j\in J}c_j^ka_j .
\end{equation}
We claim that $|c_j^k|\leq 1$ for all $j\in J$ and $k\in J'$. Indeed, fix such $j$ and $k$ and put $J_j^k := (J \cup \{k\}) \setminus \{j\}$.  By \eqref{kombin} and multilinearity of the determinant, $|\det A_{I,J_j^k}| = |c_j^k| \cdot |\det A_{I,J}|$. The maximality of $|\det A_{I,J}|$ implies
$$
|\det A_{I,J}|\geq |\det A_{I,J_j^k}|=|c_j^k||\det A_{I,J}|.
$$
Since $\det A_{I,J}\neq 0$, this gives $|c_j^k|\leq 1$, as required. Now let $k \in J'$. By \eqref{kombin}  and the above observation, we get
$$
|a_{kk}|\leq \sum_{j\in J}|c_j^k||a_{kj}|\leq \sum_{j\in J}|a_{kj}|.
$$
Since $k \notin J$, this and the lemma's assumption imply
$$
\sum_{\substack{j\in J'\\ j\neq k}}|a_{kj}|=\sum_{j\ne k}|a_{kj}|-\sum_{j\in J}|a_{kj}|<2|a_{kk}|-|a_{kk}|=|a_{kk}|.
$$
Thus the submatrix $A_{J',J'}$ is strictly diagonally dominant and hence, by Theorem~\ref{domina}, nonsingular. It follows that
$$
N-r=|J'|\leq \mathrm{rank}(A)=r.
$$
Hence $\tfrac{N}{2}\leq r$,  a contradiction.\\
~\\
The previous rank estimate can now be applied to the complementary projection. As a consequence, a small-norm projection may have only a limited number of diagonal entries that are not close to one.

\begin{lemma}\label{conajmniej}
Let $V\subset \ell_\infty^N$ be a subspace of dimension $N-m$, and let
$P=(p_{ij})_{i,j=1}^N$ be a projection onto $V$ satisfying
$\|P\|\leq 1+\varepsilon,$ for some $\varepsilon\geq0.$ Then
$|p_{ii}|\geq 1-\varepsilon$ for at least $N-2m$ indices $i\in[1:N]$.
\end{lemma}
{\sc Proof.}
Let $Q={\rm Id}-P.$ Since $P$ is a projection onto a subspace of dimension $N-m$, the
complementary projection $Q$ has rank $m$. Denote by
$ I:=\{i\in[1:N]: |p_{ii}|<1-\varepsilon\}$
the set of indices where the desired inequality does not hold. Let
$k=|I|$. We claim that $k\leq 2m$. Indeed, for every $i\in I$,
$ |q_{ii}| = |1-p_{ii}| \geq 1-|p_{ii}|> \varepsilon.$  Furthermore, from $\|P\|\leq 1+\varepsilon$ we get $ |1-q_{ii}|+\sum_{j\neq i}|q_{ij}|
\leq 1+\varepsilon. $ Combining this with the fact that
$ |1-q_{ii}|\geq 1-|q_{ii}|$ implies
$$
\sum_{j\neq i}|q_{ij}|
\leq
\varepsilon+|q_{ii}|
<2|q_{ii}|.
$$
Now let
$\widetilde{Q}=(q_{ij})_{i,j\in I}.$
 By  the latter and Lemma~\ref{rankus} 
 $$
\frac{k}{2}\leq {\rm rank}(\widetilde{Q})\leq {\rm rank(Q)}= m.
$$
Thus, $k \leq 2m$, as claimed.\\
~\\
In the codimension-two case, the previous estimate leaves at most four coordinates uncontrolled.  This leads us to the geometry of two-dimensional subspaces of $\ell_\infty^4$.
\begin{lemma}\label{lem1}
Let $f,g\in \ell_1^4$ be linearly independent and set
$ V=\ker f\cap\ker g. $
Suppose that there exists $k\in[1:4]$ such that
$ 3|\Delta_{lk}|<\|\Delta_l\|_1  $  for all $l\neq k,$
then there exists $v\in V$ such that $|v_k|=1$ and $|v_i|\leq\frac12$
for all $i\neq k.$
\end{lemma}
{\sc Proof.} 
By permuting the coordinates, we may assume that $k=4$. We first establish that
$
{\rm span}\{\Delta_1,\Delta_2,\Delta_3\}=V^\perp.
$
Indeed, each $\Delta_l$ belongs to
$
V^\perp={\rm span}\{f,g\},
$
and the inequalities
\begin{equation}\label{inq123}
3|\Delta_{l4}|<\|\Delta_l\|_1,
\qquad l=1,2,3,
\end{equation}
ensure that $\Delta_1,\Delta_2,\Delta_3$ are all nonzero. Assume that their span is a proper subspace of $V^\perp$. Since
$\dim V^\perp=2$, their span must then be one-dimensional. Thus
$\Delta_1$ is linearly dependent with both $\Delta_2$ and $\Delta_3$,
which yields
$\Delta_{12}=\Delta_{13}=0.$
As $\Delta_{11}=0$, it follows that
$\|\Delta_1\|_1=|\Delta_{14}|.$
which contradicts \eqref{inq123}. Moreover, for every $i\in[1:3]$, at least one of
$\Delta_{1i},\Delta_{2i},\Delta_{3i}$ is nonzero. Otherwise, since
$\Delta_1,\Delta_2,\Delta_3$ span $V^\perp$, we would have
$f_i=g_i=0$ for some $i$, and therefore $\Delta_i=0$, contradicting
\eqref{inq123}.
For $h\in \ell_1^4$ by $\widetilde{h}$ we denote the restriction of the functional to subspace $V.$ 
First, consider the case where $\widetilde{e_1},\widetilde{e_2},\widetilde{e_3}$ do not
span $V^*$. Since $\dim V^*=2$, their span has dimension one and consequently the intersection of kernels of these functionals has dimension one. Hence  $(0,0,0,1)\in V$ and is our desired vector.\\
Now assume that vectors $\widetilde{e_1},\widetilde{e_2},\widetilde{e_3}$ span whole space $V^{*}.$  First observe that $\widetilde{e_4}$ can not be equal to $0.$ Indeed if so,  then ${\rm span}\{\Delta_1,\Delta_2,\Delta_3\}={\rm span}\{e_4\}\subsetneq V^{\perp},$ a contradiction. We will show that $\widetilde{e_4}\notin {\rm conv}\{ \pm2\widetilde{e_1},\pm2\widetilde{e_2},\pm2\widetilde{e_3}\}=:C.$ For the contrary assume that $\widetilde{e_4}\in C.$ By the definition of set $C$ there exists $t\geq1$ such that $t\widetilde{e_4}\in\partial C.$ By Carath\'eodory theorem there exists $i,j\in[1:3]$ and $a,b\in\mathbb{R},$ such that $|a|+|b|=1$ and  $t\widetilde{e_4}=2a\widetilde{e_i}+2b\widetilde{e_j}.$ Then $te_4-2ae_i-2be_j\in V^\perp.$ Let $l\in [1:3]\setminus \{i,j\}.$ Then at least one $\Delta_{il} $ or $\Delta_{jl} $ has to be nonzero. Hence 
$$
te_4-2ae_i-2be_j=\frac{t\Delta_l}{\Delta_{l4}}
$$
and consequently 
$$
|a|+|b|=\frac{t(|\Delta_{li}|+|\Delta_{lj}|)}{2|\Delta_{l4}|}=\frac{t(\|\Delta_l\|_1-|\Delta_{l4}|)}{2|\Delta_{l4}|}>t\geq 1,
$$
a contradiction. Now, by the separation theorem there exist $v\in V=(V^*)^*$ such that 
$$
|v_4|=|\widetilde{e_4}(v)|>\max\{|h(v)|:\; h\in C\}\geq 2|\widetilde{e_l}(v)|=2|v_l|\;\; \textrm{ for all } l\in[1:3].
$$
Dividing $v$ by $|v_4|$ gives the desired vector.\\
~\\

\begin{lemma}\label{godorbed}
Let $f,g\in \ell_1^4$ be linearly independent and set
$ V=\ker f\cap\ker g. $
Suppose that there exist $i,j\in[1:4]$ such that $i\neq j$ and
$ 3|\Delta_{ij}|\leq \|\Delta_j\|_1  $ 
then there exists $v\in V$ such that $\|v\|_{\infty}=1$ and $|v_l|\leq\frac12$
for all $l\notin\{i,j\}.$
\end{lemma}
{\sc Proof.} By relabeling the coordinates, it is enough to assume that  $i=1$ and $j=2$. Suppose first that $\Delta_{12}=0$. Then the restrictions of $f$ and
$g$ to ${\rm span}\{e_1,e_2\}$ are linearly dependent, and hence there
exists a nonzero vector $v\in V$ with $v_3=v_4=0.$
After normalization $v$, we obtain $\|v\|_\infty=1$, and the conclusion
follows.\\
If $\Delta_{12}\neq 0,$ then $V=\ker \Delta_1\cap\ker \Delta_2,$ and by the definition of space V 
$$u:=(\Delta_{23},-\Delta_{13},\Delta_{12},0)\in V \;\;\;\textrm{ and } \;\;\;w:=(\Delta_{24},-\Delta_{14},0,\Delta_{12})\in V$$  
and consequently 
$$
v:=\frac{1}{|\Delta_{23}|+|\Delta_{24}|}({\rm sgn}(\Delta_{23})u+{\rm sgn}(\Delta_{24})w )\in V.
$$
Moreover $\|v\|_{\infty}\geq 1$ and for $l\notin \{1,2\} $
$$
|v_l|\leq\frac{|\Delta_{12}|}{|\Delta_{23}|+|\Delta_{24}|}=\frac{|\Delta_{12}|}{\|\Delta_2\|-|\Delta_{12}|}\leq \frac{1}{2}.
$$
once more by renorming vector $v$ we yield the thesis.\\
~\\
Combining the two previous lemmas gives the following theorem.
\begin{thm}\label{beechandwitch}
Let $f,g\in \ell_1^4$ be linearly independent and set
$ V=\ker f\cap\ker g. $
There exist two disjoint sets $I, J\subset[1:4]$ and vectors $v,w\in V$ such that 
$$
\|v\|_\infty=\|w\|_\infty=1 \;\;\;\textit{ and }\;\;\; \max_{i\notin I}|v_i|\leq\frac{1}{2}, \; \max_{i\notin J}|w_i|\leq\frac{1}{2}.
$$
\end{thm}
{\sc Proof.} For convenience, we say that a set $I\subset [1:4]$ is good if there exists vector $v\in V$ such that $\|v\|_{\infty}=1$ and $\max_{i\notin I}|v_i|\leq\frac{1}{2},$ and bad otherwise.  Assume, to the contrary, that no two disjoint sets
are good. Thus, in each of the three complementary pairs
$$
\{1,2\},\{3,4\};\qquad
\{1,3\},\{2,4\};\qquad
\{1,4\},\{2,3\},
$$
at least one set is bad. It follows that either there exist three distinct
indices $i,j,k\in[1:4]$ such that
$
\{i,j\},\;\{j,k\},\;\{k,i\}
$
are all bad, or there exists an index $i\in[1:4]$ such that
$
\{i,k\}
$
is bad for every $k\neq i$.
We begin with the first case. After relabeling the coordinates, we may
assume that $\{1,2\},\{2,3\},\{3,1\}$ are bad. By badness of $\{1,2\}$ and $\{3,1\}$ using  Lemma \ref{godorbed} we get
\begin{align*}
    \|\Delta_1\|_1<3|\Delta_{12}| 
    \;\;\;\textrm{ and }\;\;\;
    \|\Delta_1\|_1<3|\Delta_{13}|. 
\end{align*}
Adding the above inequalities yields
\begin{align*}
    2\|\Delta_1\|_1<3|\Delta_{12}|+3|\Delta_{13}|=3\|\Delta_1\|_1-3|\Delta_{14}|
\end{align*}
and consequently 
$$
3|\Delta_{14}|<\|\Delta_1\|_1.
$$
Analogously by badness of $\{1,2\}, \{2,3\}$ and $\{1,3\}, \{2,3\}$ we get
\begin{align*}
    3|\Delta_{24}|<\|\Delta_2\|_1 \;\;\;\textrm{ and }\;\;\;  3|\Delta_{34}|<\|\Delta_3\|_1.
\end{align*}
Therefore by Lemma \ref{lem1} set $\{4\}$ is good. Set $\{1,2,3\}$ also will be good because $V$ is two dimensional space, so there has to exist $v\in V$ such that $v_4=0,$ and $\|v\|_{\infty}=1,$ a contradiction.\\
Now consider the remaining case and assume that for some $i\in[1:4]$ every set $\{i,k\}$ is bad if $k\neq i.$ By Lemma \ref{godorbed} for all $k\neq i$
$$
\|\Delta_i\|_{1}<3|\Delta_{ik}|.
$$
Adding these inequalities, we obtain $3\|\Delta_i\|_1<3\|\Delta_i\|_1,$ a contradiction.\\
~\\
The preceding geometric argument, together with Lemma~\ref{conajmniej}, yields the following result for projections onto codimension-two subspaces.
\begin{thm}\label{base}
Let $N \geq 4$, and let $V \subset \ell_\infty^N$ be a subspace of dimension $N-2$.
Let $P = (p_{ij})_{i,j=1}^N : \ell_\infty^N \to V$  be a projection satisfying
$\|P\| \leq 1 + \varepsilon$ for some $ \varepsilon \geq 0 $.
Then there exist pairwise disjoint nonempty sets
$A_1,\ldots,A_{N-2}\subset[1:N]$ and indices $k_i\in A_i$ such that
$$
\sum_{l\in A_i}|p_{k_i l}|\geq1-\varepsilon,\;\;
{\it for }\;\; i\in[1:N-2].
$$
\end{thm}
{\sc Proof.}
By Lemma~\ref{conajmniej}, after relabeling the coordinates we may assume that
$
|p_{ii}|\geq1-\varepsilon,$ for $i\in[5:N].$
Hence $\{5\},\ldots,\{N\}$ already give $N-4$ of the required sets.
Set
$$
\widetilde W:=\{x\in V:x_i=0,\ i\in[5:N]\}.
$$
Since $\dim\widetilde W\geq2$, choose a two-dimensional subspace
$W\subset\widetilde W$ and identify it with its restriction to the first
four coordinates. By the Theorem~\ref{beechandwitch}, there exist disjoint nonempty sets
$I,J\subset[1:4]$ and $v,w\in W$ such that
$$
\|v\|_\infty=\|w\|_\infty=1,
\qquad
|v_i|\leq\frac12\ (i\notin I),
\qquad
|w_i|\leq\frac12\ (i\notin J).
$$
Let $i\in I$ be such that $|v_i|=1$.
As $v\in V$, we have $Pv=v$, and therefore
\begin{align*}
1
=|v_i|
&=
\Big|\sum_{l=1}^N p_{il}v_l\Big| 
\leq
\sum_{l\in I}|p_{il}||v_l|
+
\sum_{l\notin I}|p_{il}||v_l| 
\leq
\sum_{l\in I}|p_{il}|
+
\frac12\sum_{l\notin I}|p_{il}| \\
&=
\frac12\sum_{l\in I}|p_{il}|
+
\frac12\sum_{l=1}^N|p_{il}| 
\leq
\frac12\sum_{l\in I}|p_{il}|
+
\frac12(1+\varepsilon).
\end{align*}
Hence
$$
\sum_{l\in I}|p_{il}|\geq1-\varepsilon.
$$
The analogous argument gives $j\in J$ such that
$$
\sum_{l\in J}|p_{jl}|\geq1-\varepsilon.
$$
Thus $I,J,\{5\},\ldots,\{N\}$ have the required properties.
\begin{thm}\label{Bazur}
Let $V\subset\ell_\infty^N$ be a subspace of dimension $N-2$. If
$\lambda(V)<\frac43,$ then
$$
d(V,\ell_\infty^{N-2})
\leq
\frac{\lambda(V)}{4-3\lambda(V)}.
$$
\end{thm}
{\sc Proof.}
Since $\lambda(V)<\frac43,$ there exists $0\leq\varepsilon < \frac13 $ such that $\lambda(V)=1+\varepsilon.$ If $\varepsilon=0$ it is well known that  $d(V,\ell_\infty^{N-2})=1$ (see e.g. \cite{BB}) and the above statement is true. So it is enough to consider $0<\varepsilon < \frac13. $ Since $\lambda(V)=\lambda(V,\ell_{\infty}^N)$  there exists   a projection $P:\ell_\infty^N\to V$ such that
$$
\|P\|= 1+\varepsilon=\lambda(V).
$$
By Theorem~\ref{base}, there exist pairwise disjoint nonempty sets
$A_1,\ldots,A_{N-2}\subset[1:N]$ and indices $k_i\in A_i$ such that
$$
\sum_{j\in A_i}|p_{k_i j}|\geq1-\varepsilon,
\qquad i\in[1:N-2].
$$
Let $p_j$ denote the $j$-th column of $P$ and set
$$
q_i:=\sum_{j\in A_i}{\rm sgn}(p_{k_i j})p_j.
$$
Since each $p_j$ belongs to $V$, we have $q_i\in V$. Moreover,
\begin{equation}\label{qdiag}
(q_i)_{k_i}
=
\sum_{j\in A_i}{\rm sgn}(p_{k_i j})p_{k_i j}
=
\sum_{j\in A_i}|p_{k_i j}|
\geq1-\varepsilon.
\end{equation}
The disjointness of the sets $A_i$ implies that, for every
$l\in[1:N]$,
\begin{align}
\sum_{i=1}^{N-2}|(q_i)_l|
\leq
\sum_{i=1}^{N-2}\sum_{j\in A_i}|p_{lj}| 
\leq
\sum_{j=1}^N|p_{lj}|
\leq1+\varepsilon.
\label{qineq}
\end{align}
We first show that $q_1,\ldots,q_{N-2}$ are linearly independent.
Suppose that
$\sum_{i=1}^{N-2}a_iq_i=0$
and choose $m$ such that $|a_m|=\max_i|a_i|$. 
Then, using \eqref{qdiag} and \eqref{qineq}, we obtain
\begin{align*}
0 &= \Big| \sum_{i=1}^{N-2}a_i(q_i)_{k_m} \Big|
\geq |a_m||(q_m)_{k_m}|- \sum_{i\neq m}|a_i||(q_i)_{k_m}|
\geq |a_m|\Big( |(q_m)_{k_m}|- \sum_{i\neq m}|(q_i)_{k_m}|\Big)\\
&=|a_m| \Big( 2|(q_m)_{k_m}| - \sum_{i=1}^{N-2}|(q_i)_{k_m}| \Big)
\geq |a_m|\bigl(2(1-\varepsilon)-(1+\varepsilon)\bigr)
=|a_m|(1-3\varepsilon).
\end{align*}
Since $\varepsilon<\frac13$ we get $|a_m|=0,$ and consequently $a_i=0$ for all $i\in[1,N-2].$
Thus $q_1,\ldots,q_{N-2}$ are linearly independent and therefore form a basis of $V$.
Define isomorphism
$$
T:V\longrightarrow\ell_\infty^{N-2},
\qquad
T\Big(\sum_{i=1}^{N-2}a_iq_i\Big)
=
(a_1,\ldots,a_{N-2}).
$$
Let $a=(a_1,\ldots,a_{N-2})\in\ell_\infty^{N-2}$. By \eqref{qineq},
\begin{align*}
\|T^{-1}a\|_\infty
&= \Big\|\sum_{i=1}^{N-2}a_iq_i\Big\|_\infty
\leq \|a\|_\infty \max_{l\in[1:N]}\sum_{i=1}^{N-2}|(q_i)_l|
\leq
(1+\varepsilon)\|a\|_\infty.
\end{align*}
Thus
\begin{equation}\label{Tinverse}
\|T^{-1}\|\leq1+\varepsilon.
\end{equation}
On the other hand, choose $m$ such that $|a_m|=\|a\|_\infty$. Using
\eqref{qdiag} and \eqref{qineq} once again, we get
\begin{align*}
\Big\|\sum_{i=1}^{N-2}a_iq_i\Big\|_\infty
\geq \Big| \sum_{i=1}^{N-2}a_i(q_i)_{k_m} \Big|
\geq |a_m| \Big( 2|(q_m)_{k_m}|- \sum_{i=1}^{N-2}|(q_i)_{k_m}| \Big)
\geq (1-3\varepsilon)\|a\|_\infty.
\end{align*}
Consequently,
\begin{equation}\label{Tnorm}
\|T\|\leq\frac1{1-3\varepsilon}.
\end{equation}
Combining \eqref{Tinverse} and \eqref{Tnorm}, we obtain
$$
d(V,\ell_\infty^{N-2})
\leq
\|T\|\,\|T^{-1}\|
\leq
\frac{1+\varepsilon}{1-3\varepsilon}=\frac{\lambda(V)}{4-3\lambda(V)}
$$
 as required.

\section{Discussion}\label{sec:discussion}

The reasoning presented in Theorem~\ref{Bazur} can be generalized to any $n$-dimensional
subspace $X\subset\ell_\infty^N$, provided that an analog of
Theorem~\ref{base} is available. Specifically, it is sufficient to show
that every projection $P=(p_{ij})_{i,j=1}^N$ onto $X$ with
$\|P\|\leq1+\varepsilon<\frac43$ admits pairwise disjoint nonempty sets
$A_1,\ldots,A_n\subset[1:N]$ and indices $k_i\in A_i$ satisfying
\begin{equation}\label{colcomb}
\sum_{j\in A_i}|p_{k_i j}|\geq1-\varepsilon,
\qquad i\in[1:n].
\end{equation}
The remainder of the proof of Theorem~\ref{Bazur} then proceeds without
modification and yields
$$
d(X,\ell_\infty^n)
\leq
\frac{1+\varepsilon}{1-3\varepsilon}.
$$
For subspaces of codimension two, this result is precisely the content of
Theorem~\ref{base}, the proof of which relies on the geometric result
Theorem~\ref{beechandwitch}. In the case of hyperplanes, the 
analog of this result is available.
\begin{prop}\label{hypergeom}
Let $V=\ker f\subset\ell_\infty^N$, where
$0\neq f=(f_1,\ldots,f_N)\in\ell_1^N$. Then there exist
$N-1$ pairwise disjoint nonempty sets
$A_1,\ldots,A_{N-1}\subset[1:N]$ and vectors
$v^1,\ldots,v^{N-1}\in V$ such that
\begin{equation}\label{HilaryHahn}
\|v^i\|_\infty=1
\qquad\text{and}\qquad
|v^i_j|\leq\frac12
\quad\text{for }j\notin A_i.
\end{equation}
\end{prop}
{\sc Proof.}
We may assume that $\|f\|_1=1$. First, we will show that if for some $i$
$3|f_i|\leq1,$ then $\{i\}$ has property \eqref{HilaryHahn}.
If $f_i=0$, then $e_i\in V$ and there is nothing to prove. Otherwise,
since $|f_i|\leq\frac12\sum_{j\neq i}|f_j|$ the vector
$ e_i-f_i\Big(\sum_{k\neq i}|f_k|\Big)^{-1}
\sum_{j\neq i}{\rm sgn}(f_j)e_j$
belongs to $V,$ and all coordinates except the $i$-th have
modulus at most $1/2$. Since $\sum_i|f_i|=1$, there can be at most two indices for which
$3|f_i|>1$. If there is at most one such index, we may take
$N-1$ singletons. If there are two, say $j$ and $k$, then all the
remaining singletons have the required property, while $\{j,k\}$ does
as well, since $f_ke_j-f_je_k\in V\setminus\{0\}$ and this vector is supported on $\{j,k\}$.\\
~\\
Combining this proposition with the initial observation in this
section gives
$$
d(V,\ell_\infty^{N-1})
\leq
\frac{\lambda(V)}{4-3\lambda(V)},
$$
for hyperplanes with $\lambda(V)<\frac43.$
It is natural to ask whether the geometric statement used in the proof
of Theorem~\ref{base} extends to higher codimensions. The following
example shows that this is not the case. In fact, its natural analog
fails for $n$-dimensional subspaces of $\ell_\infty^{2n}$ when
$n\geq3$.
\begin{prop}\label{counterhigher}
For every $n\geq3$ there exists an $n$-dimensional subspace
$W_n\subset\ell_\infty^{2n}$ such that there are no pairwise disjoint
nonempty sets $A_1,\ldots,A_n\subset[1:2n]$ and vectors
$w_1,\ldots,w_n\in W_n$ satisfying
$$
\|w_i\|_\infty=1
\qquad\text{and}\qquad
|w_i^j|\leq\frac12
\quad\text{for }j\notin A_i.
$$
\end{prop}
{\sc Proof.} Set vectors
\begin{equation*}
    u:=(1,\tfrac{4}{5},\tfrac{1}{4},-\tfrac{1}{4},-\tfrac{4}{5})\qquad {\rm and }\qquad v:=(0,\tfrac{3}{5},1,1,\tfrac{3}{5}) 
\end{equation*} and define space $V:={\rm span}\{u,v\}.$ We claim that every nonzero
$x\in V$ has at least three coordinates of modulus greater than
$\frac12\|x\|_\infty$. Let
$$
x:=au+bv =
\Big( a,\frac45a+\frac35b,\frac14a+b,
-\frac14a+b,-\frac45a+\frac35b \Big).
$$
Since changing the sign of either coefficient only permutes the
absolute values of the coordinates, we may assume that $a,b\geq0$.
Then $|x_4|\leq x_3$ and $|x_5|\leq x_2$, so
$\|x\|_\infty=\max\{x_1,x_2,x_3\}.$
Suppose that $\|x\|_{\infty}=|x_1|=x_1.$ Observe that $x_2-x_5=\tfrac{8}{5}x_1,$ and consequently
$$
\frac{8}{5}\|x\|_{\infty}=|x_2-x_5|\leq |x_2|+|x_5|\leq\|x\|_{\infty}+\min\{|x_2|,|x_5|\},
$$
which implies $|x_1|,|x_2|,|x_5|>\tfrac{1}{2}\|x\|_{\infty}.$\\
Assume now that $\|x\|_{\infty}=|x_2|.$ Observe that $13x_1+12x_3=20x_2,$ and consequently
$$
20\|x\|_{\infty}\leq 13\|x\|_{\infty}+12\min\{|x_1|,|x_3|\}
$$
which implies $|x_1|,|x_2|,|x_3|>\tfrac{1}{2}\|x\|_{\infty}.$\\
Finally let $\|x\|_{\infty}=|x_3|.$ Since $x_1,x_3\geq 0$, by the previous equality $|x_2|\geq\tfrac{12}{20}\|x\|_{\infty}.$ Moreover, $x_3=\frac{1}{2}x_1+x_4,$ and consequently $|x_3|\leq \frac{3}{2}\max\{|x_1|,|x_4|\},$ Hence $x_3$, $x_2$, and at least one of $x_1,x_4$ have modulus larger than $\tfrac{1}{2}\|x\|_\infty.$\\
For $n\geq3$, we set
$$
W_n:=\{(x,s,t_1,t_1,\ldots,t_{n-3},t_{n-3}):
       x\in V,\ s,t_1,\ldots,t_{n-3}\in\mathbb R\}.
$$
Clearly, $\dim W_n=n$. Suppose that sets $A_i$ and vectors $w^i$
satisfy the stated conditions. Choose $k_i\in A_i$ such that
$|w^i_{k_i}|=1$. Since the sets $A_i$ are pairwise disjoint,
the indices $k_1,\ldots,k_n$ are distinct.
If $k_i$ lies in one of the repeated pairs, both coordinates of that pair have modulus one in $w^i$ and must therefore belong to $A_i$. Consequently, each repeated pair contains at most one selected index. There are $n-3$ such pairs, and the only other coordinate outside $[1:5]$ is $6$. Hence at most $n-2$ of the selected indices lie outside $[1:5]$, so at least two lie in $[1:5]$.
The restrictions of the corresponding vectors to the first five
coordinates belong to $V$ and have norm one. Each therefore has
at least three coordinates of modulus greater than $1/2$, all
of which must belong to the corresponding set $A_i$. This is
impossible, since two disjoint subsets of $[1:5]$ cannot both
contain at least three elements.\\
~\\
Proposition~\ref{counterhigher} shows that the geometric argument used
in the proof of Theorem~\ref{base} does not extend to higher codimension.
It remains possible, however, that the conclusion of
Theorem~\ref{base}, which is all that is needed in the proof of
Theorem~\ref{Bazur}, continues to hold. For the spaces $W_n$ constructed
above this is indeed the case.\\
~\\
We keep the notation from the preceding proof.
\begin{prop}\label{combiningWn}
Let $n\geq3$, and let
$P=(p_{ij}):\ell_\infty^{2n}\to W_n$
be a projection satisfying
$$
\|P\|\leq1+\varepsilon,
\qquad 0\leq\varepsilon\leq\frac13.
$$
Then there exist pairwise disjoint nonempty sets
$A_1,\ldots,A_n\subset[1:2n]$ and indices $k_i\in A_i$ such that
$$
\sum_{j\in A_i}|p_{k_i j}|\geq1-\varepsilon,
\qquad i\in[1:n].
$$
\end{prop}
{\sc Proof.} Consider the submatrix $\widetilde P=(p_{ij})_{i,j=1}^5$. Observe that for $y\in\ell_\infty^5$,  $P(y,0,\ldots,0)\in W_n$, and consequently its first five
coordinates belong to $V$. These coordinates are precisely
$\widetilde P y$, and hence $\widetilde P$ maps $\ell_\infty^5$ into $V.$ Moreover, for every $x\in V$, the vector $(x,0,\ldots,0)$
belongs to $W_n$, and therefore
$P(x,0,\ldots,0)=(x,0,\ldots,0).$
Taking the first five coordinates gives $\widetilde Px=x$.
Thus $\widetilde P$ is a projection onto $V$, and
$\|\widetilde P\|\leq\|P\|\leq1+\varepsilon.$
We first show that
\begin{equation}\label{pairalternative}
\max\{|p_{21}|+|p_{22}|,\ |p_{51}|+|p_{55}|\}
\geq1-\varepsilon.
\end{equation}
Recall that $V={\rm span}\{u,v\}$ where
$$
 u=(1,\tfrac{4}{5},\tfrac{1}{4},-\tfrac{1}{4},-\tfrac{4}{5})\qquad {\rm and }\qquad v=(0,\tfrac{3}{5},1,1,\tfrac{3}{5}). 
 $$
For $z=(1,1,1,-1,-1)$, direct calculation gives
\begin{align*}
u+\tfrac13v-\tfrac25(e_1+e_2)-\tfrac14z
&=\Big(\tfrac7{20},\tfrac7{20},\tfrac13,\tfrac13,-\tfrac7{20}\Big),\\
-u+\tfrac13v-\tfrac25(e_5-e_1)+\tfrac14z
&=\Big(-\tfrac7{20},-\tfrac7{20},\tfrac13,\tfrac13,\tfrac7{20}\Big).
\end{align*}
Since $\|\widetilde P\|\leq1+\varepsilon$, the image of each of
these vectors under $\widetilde P$ has all coordinates bounded
in modulus by $\frac7{20}(1+\varepsilon)$. Thus, taking the second
coordinate of the first image and the fifth coordinate of the second,
we obtain
 \begin{align*}
1&\leq
\frac25(p_{21}+p_{22})
+\frac14(\widetilde Pz)_2+\frac7{20}(1+\varepsilon),\\
1
&\leq
\frac25(-p_{51}+p_{55})
-\frac14(\widetilde Pz)_5+\frac7{20}(1+\varepsilon).
\end{align*}
Since $\widetilde Pz\in V$, the definition of $V$ gives
$$
(\widetilde Pz)_2-(\widetilde Pz)_5
=\frac85(\widetilde Pz)_1.
$$
Moreover, $\|z\|_\infty=1$, so
$|(\widetilde Pz)_1|\leq1+\varepsilon$.
Adding the preceding inequalities therefore yields
\begin{align*}
2
&\leq
\frac25(p_{21}+p_{22}-p_{51}+p_{55})
+\frac25(\widetilde Pz)_1
+\frac7{10}(1+\varepsilon)\\
&\leq
\frac25(p_{21}+p_{22}-p_{51}+p_{55})
+\frac{11}{10}(1+\varepsilon).
\end{align*}
Consequently,
$$
p_{21}+p_{22}-p_{51}+p_{55}
\geq\frac{9-11\varepsilon}{4}
\geq2(1-\varepsilon),
$$
where the last inequality follows from $\varepsilon\leq1/3$.
Hence at least one of $p_{21}+p_{22}$ and $-p_{51}+p_{55}$
is at least $1-\varepsilon$, which proves \eqref{pairalternative}.\\
If $|p_{21}|+|p_{22}|\geq1-\varepsilon$, take
$$
A_1=\{1,2\},\quad k_1=2,
\qquad
A_2=\{3,4,5\},\quad k_2=4,
$$
and set
$$
x:=-u+3v
=
\Big(-1,1,\frac{11}{4},\frac{13}{4},\frac{13}{5}\Big).
$$
If instead $|p_{51}|+|p_{55}|\geq1-\varepsilon$, we take
$$
A_1=\{1,5\},\quad k_1=5,
\qquad
A_2=\{2,3,4\},\quad k_2=3,
$$
and set
$$
x:=u+3v
=
\Big(1,\frac{13}{5},\frac{13}{4},\frac{11}{4},1\Big).
$$
In either case,
$$
\sum_{j\in A_1}|p_{k_1j}|\geq1-\varepsilon.
$$
Moreover, $A_1$ and $A_2$ form a partition of $[1:5]$, and
$$
x\in V,\qquad
x_{k_2}=\|x\|_\infty=\frac{13}{4},
\qquad
|x_j|=1\quad\text{for }j\in A_1.
$$
Since $\widetilde Px=x$, we have
\begin{align*}
\frac{13}{4}
=
\Big|\sum_{j=1}^5p_{k_2j}x_j\Big|
&\leq
\sum_{j\in A_1}|p_{k_2j}|
+
\frac{13}{4}\sum_{j\in A_2}|p_{k_2j}|
=
\sum_{j=1}^5|p_{k_2j}|
+
\frac94\sum_{j\in A_2}|p_{k_2j}|\\
&\leq
1+\varepsilon
+
\frac94\sum_{j\in A_2}|p_{k_2j}|.
\end{align*}
Therefore
$\sum_{j\in A_2}|p_{k_2j}|
\geq
1-\frac49\varepsilon
\geq
1-\varepsilon.$ It remains to choose the remaining sets. Let
$A_3=\{6\},$ $k_3=6,$
and, for $j\in[1:n-3]$, set
$A_{3+j}=\{2j+5,2j+6\},$ $
k_{3+j}=2j+5.$
Since $e_6\in W_n$ and $e_{2j+5}+e_{2j+6}\in W_n,$
and $P$ is a projection onto $W_n$, we have
$$
Pe_6=e_6\qquad \textrm{and} \qquad P(e_{2j+5}+e_{2j+6})
=
e_{2j+5}+e_{2j+6}.
$$
In particular,
$$
p_{66}=1\qquad \textrm{and} \qquad
p_{2j+5,\,2j+5}+p_{2j+5,\,2j+6}=1.
$$
Hence
$\sum_{l\in A_3}|p_{k_3l}|=1$
and, for every $j\in[1:n-3]$,
$$
\sum_{l\in A_{3+j}}|p_{k_{3+j}l}|
=
|p_{2j+5,\,2j+5}|+|p_{2j+5,\,2j+6}|
\geq1.
$$
Thus all the sets $A_1,\ldots,A_n$ satisfy the required estimate,
which completes the proof.\\
~\\
The preceding proposition is not vacuous, since these spaces admit
projections of norm strictly smaller than $4/3$.
\begin{prop}\label{smallprojectionWn}
For every $n\geq3$ there exists a projection
$P_n:\ell_\infty^{2n}\to W_n$ such that
$\|P_n\|=\tfrac{13}{10}<\tfrac43.$
\end{prop}
{\sc Proof.}
Define
$
f:=(
\tfrac23,\tfrac16,\tfrac2{15},-\tfrac2{15},-\tfrac16
),$
and $g:=(0,\frac16,\frac25,\frac25,\frac16).$
One can easily check, that
$f(u)=1,$  $f(v)=0,$ 
$g(u)=0,$ and $g(v)=1.$
Hence the operator
$Q(\cdot)=f(\cdot)u+g(\cdot)v$
is a projection from $\ell_\infty^5$ onto $V.$  The $i$-th row of $Q$
is $u_i f+v_i g$, and the sums of the absolute values of the entries
in its five rows are, respectively,
$\tfrac{190}{150}$, $\tfrac{192}{150}$, $\tfrac{195}{150}$,
$\tfrac{195}{150}$, and $\tfrac{192}{150}$. Hence
$\|Q\|=\tfrac{195}{150}=\tfrac{13}{10}.$ For $y\in\ell_\infty^5$ and
$s,a_1,b_1,\ldots,a_{n-3},b_{n-3}\in\mathbb R$, define
\begin{align*}
P_n(y,s,a_1,b_1,\ldots,a_{n-3},b_{n-3})
&:=
\Big(
Qy,s,
\tfrac{a_1+b_1}{2},\tfrac{a_1+b_1}{2},
\ldots,
\tfrac{a_{n-3}+b_{n-3}}{2},
\tfrac{a_{n-3}+b_{n-3}}{2}
\Big).
\end{align*}
Since $Q$ is a projection onto $V$, it follows directly that $P_n$
is a projection onto $W_n$. Moreover, since
$|(a_j+b_j)/2|\leq\max\{|a_j|,|b_j|\}$ and the restriction of $P_n$
to the first five coordinates is $Q$, we have
$$
\|P_n\|=\max\{\|Q\|,1\}=\frac{13}{10}<\frac43.
$$

~\\
The spaces $W_n$ show that the geometric property used above
and the column-combining condition are genuinely different.
The former fails for these spaces, whereas \eqref{colcomb} holds
for every projection onto $W_n$ of norm at most $4/3$.

So far, we have not found a counterexample to \eqref{colcomb}
among projections of norm less than $4/3$, in any codimension.
It is therefore natural to ask whether this condition holds
for every such projection onto a subspace of $\ell_\infty^N$.
By the argument of Theorem~\ref{Bazur}, this would yield the
corresponding Banach--Mazur estimate without any restriction
on the codimension. Proposition~\ref{counterhigher} shows,
however, that a general proof would require more than a direct
extension of the geometric argument used here.

\section*{Use of AI tools}

The authors used ChatGPT (OpenAI) to assist with literature searches, the exploration
of possible proof strategies, and the revision of the exposition.
All formal statements, arguments, and proofs in the manuscript were written and checked by the authors, who take full responsibility for the accuracy and integrity of the article.

\end{document}